\documentclass[11pt]{article}
\newlength{\dinwidth}
\newlength{\dinmargin}
\usepackage{amsmath,amssymb,amsfonts}
\usepackage[alphabetic,y2k,initials]{amsrefs}

\usepackage{color, graphics}
\usepackage{float}
\usepackage{tikz}
\usepackage{pdfsync}
\usepackage{hyperref}
\usepackage{multirow}

\usepackage{longtable}

\usepackage{calrsfs}
\DeclareMathAlphabet{\pazocal}{OMS}{zplm}{m}{n}

\newcommand{\Q}{\pazocal{Q}}
\newcommand{\E}{\pazocal{E}}

\newcommand{\Pol}{\mathcal{P}}
\newcommand{\T}{\mathcal{T}}

\newcommand{\refeq}{\eqref}

\newtheorem{theorem}{Theorem}[section]

\newtheorem{corollary}[theorem]{Corollary}
\newtheorem{lemma}[theorem]{Lemma}

\newtheorem{example}[theorem]{Example}

\newenvironment{proof}[1][Proof]{\noindent\textit{#1.} }{\hfill$\Box$\newline\medskip}

\numberwithin{equation}{section}

\usepackage{authblk}
\author[1]{George E. Andrews}
\author[2,4]{Vladimir Dragovi\'c}
\author[3,4]{Milena Radnovi\'c}
\affil[1]{\textsc{Penn State University, Department Of Mathematics}}
\affil[2]{\textsc{The University of Texas at Dallas, Department Of Mathematical Sciences}}
\affil[3]{\textsc{The University of Sydney, School of Mathematics and Statistics}}
\affil[4]{\textsc{Mathematical Institute SANU, Belgrade}}
\affil[ ]{\texttt{gea1@psu.edu, vladimir.dragovic@utdallas.edu, milena.radnovic@sydney.edu.au}}

\date{}

\title{Combinatorics of periodic ellipsoidal billiards}

\begin{document}

\maketitle

\begin{abstract}
We study combinatorics of billiard partitions which arose recently in the description of periodic trajectories of ellipsoidal billiards in $d$-dimensional Euclidean and pseudo-Euclidean spaces. Such partitions uniquely codify the sets of caustics, up to their types, which generate
periodic trajectories. The period of  a periodic trajectory is the largest part while the winding numbers are the remaining summands of the corresponding partition. In order to take into account the types of caustics as well, we introduce weighted partitions. We provide closed forms for the generating functions of these partitions.

\

\noindent
\textsc{MSC2010 numbers}: \texttt{05A15, 05A17, 14H70, 37J35, 26C05}

\end{abstract}
	
\tableofcontents

\section{Introduction}
The aim of this paper is to study combinatorics of billiard partitions which arose recently in the description of periodic trajectories of ellipsoidal billiards in $d$-dimensional Euclidean and pseudo-Euclidean spaces. Such partitions uniquely codify the sets of caustics, up to their types, which generate
periodic trajectories. The period of  a periodic trajectory is the largest part while the winding numbers are the remaining summands of the corresponding partition. This correspondence and the main properties of the partitions in Euclidean case were established in the recent paper \cite{DragRadn2018}. The key moment in \cite{DragRadn2018} was relating the periodicity condition for billiard trajectories in $d$-dimensional space with the Pell functional-polynomial equation for extremal polynomials on $d$ intervals on the real line. In this relationship, the winding numbers of a periodic trajectory correspond to the numbers of internal extremal points of an extremal polynomial on nested intervals, see Theorem \ref{th:KLN} on $(E, m)$-representations below, where $E$ is the union of $d$ intervals and $m$ is the period of a billiard trajectory.In order to take into account the types of caustics as well, using Lemma \ref{lemma:audin} below, we introduce weighted partitions. We provide closed forms for the generating functions of such partitions, see Theorem \ref{th:generatingE}. Section \ref{section:PE} is devoted to the study of partitions related to ellipsoidal billiards in pseudo-Euclidean spaces and we distinguish space-type, time-type, and light-type partitions. We provide the  generating functions for those partitions in Theorem \ref{th:generatingPE}.

\subsection{$(E, m)$-representation}

For $e_0<e_1<f_1<e_2<\dots < e_g<f_g<f$, we suppose a polynomial $P_{2n}(z)$ of degree $2n$ is given, such that it is positive on $E=[e_0, f]\setminus \cup_{j=1}^g(e_j, f_j)$.
We say that such a polynomial admits \emph{an $(E, m)$-representation} if it can be represented in the form
\begin{equation}\label{eq:genPell}
P_{2n}(z)=A_{m}^2(z) + (z-e_0)(f-z)\prod_{j=1}^g\left((z-e_j)(z-f_j)\right)\cdot B_{m-g-1}^2(z),
\end{equation}
where $A_m$ and $B_{m-g-1}$ are polynomials of degrees $m$ and $m-g-1$ respectively, such that all zeros of $A_m(z)$ and
$(z-e_0)(z-f_1)\cdots(z-f_g)B_{m-g-1}(z)$ are simple, belong to $E$ and the zeros of the two polynomials interlace.

\begin{theorem}[Krein, Levin, Nudelman]\label{th:KLN}
A polynomial $P_{2n}(z)$ of degree $2n$, which is positive on $E=[e_0, f]\setminus \cup_{j=1}^g(e_j, f_j)$, admits an $(E, m)$-representation if
and only if the numbers $w_1$, \dots, $w_g$ determined by the relations:
$$
\frac{1}{2\pi}\int_E\frac{x^j \ln P_{2n}(x)}{\sqrt{T(x)}}d\,x=\sum_{k=1}^g(-1)^kw_k\int_{e_k}^{f_k}\frac{x^j}{\sqrt{|T(x)|}}d\,x+(-1)^{g+1}m\int_f^\infty \frac{x^j}{\sqrt{|T(x)|}}d\,x,
\quad
0\le j\le g-1,
$$
are positive integers.

If that is satisfied, then $w_k$ will be the number of zeros of $B_{m-g-1}(z)$ in $(e_0, e_k)$, where $B_{m-g-1}$ is the polynomial from \refeq{eq:genPell}.
\end{theorem}
We will consider the situation when $P_{2n}(x)=P_0(x)\equiv 1$. The equation \eqref{eq:genPell} then reduces to the Pell equation.
The intervals $(e_0, e_k)$ for $k=1, \dots, g$ are nested and the following relations are satisfied:
$$
|A_m(e_0)|=|A_m(f)|=|A_m(e_i)|=|A_m(f_j)|=1, \quad B_{m-g-1}\mid A_m'.
$$
Thus
we see that the numbers $w_k$ are the numbers of internal extremal points of the polynomial $A_m$ on the nested intervals, so these numbers are strictly increasing.

\subsection{Confocal quadrics in Euclidean space $\mathbf{R}^{d}$}

Let an ellipsoid be given by:
$$
\E\ :\ \frac{x_1^2}{a_1}+\dots+\frac{x_d^2}{a_d}=1,
\quad
0<a_1<a_2<\dots<a_d.
$$
The family of quadrics confocal with $\E$ is:
\begin{equation}\label{eq:confocald}
\Q_{\lambda}(x)=\frac {x_1^2}{a_1-\lambda}+\dots + \frac
{x_d^2}{a_d-\lambda}=1.
\end{equation}

For a point given by its Cartesian coordinates $x=(x_1, \dots, x_d)$, the Jacobi elliptic coordinates $(\lambda_1,\dots, \lambda_d)$ are defined as the solutions of the equation in $\lambda$: $ \Q_{\lambda}(x)=1$.
The correspondence between the elliptic and Cartesian coordinates is not injective, since points symmetric with respect to coordinate hyper-planes have equal elliptic coordinates.

The Chasles theorem states that almost any line $\ell$ in the space $\mathbf E^d$ is tangent to exactly $d-1$ non-degenerate
quadrics from the confocal family.
Moreover, any line $\ell'$ obtained from $\ell$ by a billiard reflection off a quadric from the confocal family \eqref{eq:confocald}
is touching the same $d-1$ quadrics. These $d-1$ quadrics are {\it the caustics} of a given billiard trajectory, see \cite{DragRadn2011book} and references therein. The existence of $d-1$ caustics is a geometric manifestation of integrability of billiards within quadrics.
If those quadrics are $\Q_{\alpha_1}$, \dots, $\Q_{\alpha_{d-1}}$, then
the Jacobi elliptic coordinates $(\lambda_1,\dots, \lambda_d)$ of any point on $\ell$ satisfy the
inequalities $\Pol(\lambda_j)\ge 0$, $j=1,\dots,d$, where
$$
\Pol(x)=(a_1-x)\dots(a_d-x)(\alpha_1-x)\dots(\alpha_{d-1}-x).
$$
Let $b_1<b_2<\dots<b_{2d-1}$ be constants such that
$$\{b_1,\dots,b_{2d-1}\}=\{a_1,\dots,a_d,\alpha_1,\dots,\alpha_{d-1}\}.$$
Here, clearly, $b_{2d-1}=a_d$.
The possible arrangements of the parameters $\alpha_1$, \dots, $\alpha_{d-1}$ of the caustics and the parameters $a_1$, \dots, $a_d$ of the confocal family can be obtained from the following lemma.

\begin{lemma}[\cite{Audin1994}]\label{lemma:audin}
If $\alpha_1<\alpha_2< \dots<\alpha_{d-1}$, then
$\alpha_j\in\{b_{2j-1},b_{2j}\}$, for $1\le j\le d-1$.
\end{lemma}

If $\ell$ is the line containing a segment of a billiard trajectory within $\E$, then $b_1>0$.

Along a billiard trajectory, the Jacobi elliptic coordinates satisfy:
$$
b_0=0\le\lambda_1\le b_1,
\quad
b_2\le\lambda_2\le b_3,
\quad\dots,\quad
b_{2d-2}\le\lambda_{d}\le b_{2d-1}.
$$
Moreover, along the trajectory, each Jacobi coordinate $\lambda_j$ fills the whole interval $[b_{2j-2},b_{2j-1}]$, with local extreme points being only the end-points of the interval.
Thus, $\lambda_j$ takes values $b_{2j-2}$ and $b_{2j-1}$ alternately and changes monotonously between them.
Let $\T$ be a periodic billiard trajectory and denote by $m_j$ the number of its points where $\lambda_j=b_{2j-2}$.
Based on the previous discussion, $m_j$ is also the number of the points where $\lambda_j=b_{2j-1}$.

Notice that the value $\lambda_1=0$ corresponds to an impact with the boundary ellipsoid $\E$,
value $\lambda_j=\alpha_{k}$ corresponds to a tangency with the caustic $\Q_{\alpha_k}$, and
$\lambda_j=a_k$ corresponds to an intersection with the coordinate hyperplane $x_k=0$.
Since each periodic trajectory intersects any hyperplane even number of times, we get that $m_j$ must be even whenever $b_{2j-2}$ or $b_{2j-1}$ is in the set $\{a_1,\dots, a_d\}$. We can reformulate the last statement and say
\begin{lemma}[see \cite{DragRadn2018}]\label{lemma:odd}
If $m_j$ is odd then both $b_{2j-2}$ and $b_{2j-1}$ are in the set $\{\alpha_1,\dots, \alpha_{d-1}\}$.
\end{lemma}

Following \cite{RR2014} (see also \cites{CRR2011, CRR2012}), we denote $m_0=n$, $m_d=0$, and call
$(m_0, m_1, \dots, m_{d-1})$
\emph{the winding numbers} of a given periodic billiard trajectory with period $n$. It was proven in \cite{DragRadn2018} that $m_0>m_1>\dots>m_{d-1}$. The proof was based on Theorem \ref{th:KLN} for $P_0(x)\equiv 1$ and the identifications of $g=d-1$ and $m_0=n=m$, $m_i=w_{d-i}$. Further, based on Lemma \ref{lemma:audin}, Theorem \ref{th:KLN}, and \cite{PS1999}, it was proven in \cite{DragRadn2018}, that a  partition $m_0>m_1>\dots>m_{d-1}$, with $m_{d-1}$ even and not two consecutive $m_i, m_{i+1}$ being both odd, uniquely determines the set of caustics of a given type, which generate periodic trajectories of period $m_0$ and winding numbers $m_i$.  We will refer to such partitions as the Euclidean billiard partitions.

\subsection{Weights}

 We want to go even a step further, and to assign a weight to a Euclidean billiard partition to count the number of possibilities for types of the caustics compatible with the given partition. The key information about the possible types of caustics
comes from  Lemma \ref{lemma:audin}. Namely, if all elements of a partition are even, then there are $2^{d-1}$ possible different choices of types of caustics. If the period is odd, then the number of choices is two times less, since in this case it has to be $b_1=\alpha_1$. In other words,
for odd-periodic trajectories, the caustic $\Q_{\alpha_1}$ has to be an ellipsoid. If any other winding number is odd, then two caustics have to be
of a given, fixed type. This means that an odd winding number distinct from the period reduces the number of choices for geometric types
of caustics four times. Thus we introduce the following weight function $\phi (n, d, \pi)$ for a given Euclidean billiard partition $\pi$ of length $d$ with the largest part equal $n$:
\begin{align}\label{eq:weights1}
\phi(2m, d, \pi)&=2^{d-1-2s};\\
\phi(2m+1, d, \pi)&=2^{d-2s},
\label{eq:weights2}
\end{align}
where $s$ is the total number of odd parts in $\pi$.

\begin{example}
Let $d=2$. Then $\phi(2m+1, 2, \pi)=1$ and $\phi( 2m, 2, \pi)=2$.
\end{example}
\begin{example} Let $d=3$.
For $n=4$, the only possible partition is $\pi=(4, 3, 2)$ and we have $\phi(4, 3, \pi)=1$.

For $n=5$, again there is only one partition $\pi=(5, 4, 2)$, for which $\phi(5, 3, \pi)=2$.

For $n=6$, the partitions are:
$$
\pi_1=(6, 5, 4), \quad
\pi_2=(6, 5, 2), \quad
\pi_3=(6, 4, 2), \quad
\pi_4=(6, 3, 2),
$$
and we have:
$$
\phi(6, 3, \pi_1)=\phi(6, 3, \pi_2)=\phi(6, 3, \pi_4)=1,
\quad
\phi(6, 3, \pi_3)=4.
$$

For $n=7$, the partitions are:
$$
\pi_1=(7, 6, 4), \quad
\pi_2=(7, 6, 2), \quad
\pi_3=(7, 4, 2),
$$
with
$$
\phi(7, 3, \pi_1)=\phi(7, 3, \pi_2)=\phi(7, 3, \pi_3)=2.
$$

\end{example}

In the next Section \ref{sec:ebp} we are going to study further combinatorial properties of such partitions, including the case of weighted partitions.

\section{Euclidean Billiard Partitions}\label{sec:ebp}

Let $\mathcal D$  denote the set of all integer partitions into distinct parts where
\begin{itemize}
\item [(E1)] the smallest part is even;
\item [(E2)] adjacent parts are never both odd.
\end{itemize}
Let $p_{\mathcal D}(n)$ denote the number of partitions of $n$ that are in $\mathcal D$.
\begin{align*}
1+\sum_{n\ge1}p_{\mathcal D}(n)q^n = 1&+q^2+q^4+q^5+2q^6+q^7+2q^8+3q^9+3q^{10}\\&+4q^{11}+4q^{12}+6q^{13}
+5q^{14}+9q^{15}+\dots .
\end{align*}
\begin{example}\label{ex:partnumber}
Thus, $p_{\mathcal D}(15)=9$, and the nine partitions of $15$ are $13+2$, $11+4$, $10+3+2$, $9+6$, $9+4+2$, $8+5+2$, $7+6+2$, $6+5+4$, $6+4+3+2$.
\end{example}
Additionally, we shall also need to consider weighting the partitions in $\mathcal D$ as follows. Suppose $\pi \in \mathcal D$ and that $\pi$ has $d$ parts with largest part $n$ and $s$ odd parts. The weight $\phi(n, d, \pi)$ is given by \eqref{eq:weights1} and \eqref{eq:weights2}.

Let  $p_{\mathcal D}(m, n)$ denote the number of partitions of $n$ in $\mathcal D$ with weight $m$. Then:
\begin{align*}
1+\sum_{n\ge1}p_{\mathcal D}(m, n)x^mq^n = 1&+q^2+q^4+q^5+(1+x)q^6+q^7+(1+x)q^8+3q^9\\&+(1+2x)q^{10}+(3+x)q^{11}+(1+2x+x^2)q^{12}\\&+(5+x)q^{13}
+(2+3x+x^2)q^{14}+(6+3x)q^{15}+\dots .
\end{align*}
Referring back to Example \ref{ex:partnumber}, we see that three partitions of $15$ have weight $1$, namely $2+4+9$, $2+6+7$, and $2+3+4+6$ while the remaining six have weight $0$. Thus yielding $(6+3x)$ as the coefficient of $q^{15}$.

Our object is to provide a closed form for these generating functions. To this end we begin by identifying a subset of $\mathcal D$ which we will call the irreducible partitions of  $\mathcal D$ and denote $\bar {\mathcal D}$. By irreducible we mean that $\pi \in  \bar {\mathcal D}$ if no summand of $\pi$ can be reduced by 2 with the resulting partition remaining in ${\mathcal D}$. For example $2+4+7$ is not  irreducible because
$2+4+(7-2)=2+4+5$ is still in ${\mathcal D}$. On the other hand $2+4+5$ is irreducible because $2+4+3$ destroys the order of the parts.

\begin{lemma} The elements of $\bar {\mathcal D}$ consists of those elements of ${\mathcal D}$   where
\begin{itemize}
\item [(IE1)] the smallest part is 2;
\item [(IE2)] adjacent parts are never both odd;
\item[(IE3)] the difference between adjacent parts is less or equal 2.
\end{itemize}
\end{lemma}
\begin{proof} If $\pi\in \mathcal D$ and its smallest part is $\lambda \ne 2$. Since it is even, $\lambda \ge 4$.
Thus $\lambda-2$ is even and $\lambda-2\ge 2$. Hence, we have produced a new element of $\mathcal D$ by replacing $\lambda$ by $\lambda -2$.
Thus, it has to be $\lambda=2$.

Next, suppose $\mathcal D \ni\pi=\dots+\lambda_i + \lambda_{i+1}+\dots$ and $\lambda_i - \lambda_{i+1}>2$. Then the partition obtained by replacing
$\lambda_i$ by $\lambda_i-2$ is also in $\mathcal D$: $\lambda_i - 2> \lambda_{i+1}$ and $\lambda_i - 2$ and $\lambda_{i}$ are of the same parity.
Thus, consecutive parts in partitions in $\bar {\mathcal D}$ must differ by not more than 2.
\end{proof}

\begin{corollary}\label{corollary:decomp} Every partition  $\pi\in \mathcal D$ with $d$ parts can be uniquely represented by $\pi_1+\pi_2$ where $\pi_1\in \bar {\mathcal D}$ and $\pi_2$ is a partition with not more than $d$ parts each even and conversely.
\end{corollary}
\begin{example}
Consider $9+4+2\in \mathcal D$: $ 9+4+2=(5+4+2)+(4+0+0).$
\end{example}

\begin{proof} Clearly any partition of the given form $\pi_1+\pi_2$ is in $\mathcal D$. Now suppose we are given
$$
\pi: \lambda_1+\lambda_2+\dots+\lambda_d\in \mathcal D,
$$
with each $\lambda_i>\lambda_{i+1}$. Since $\lambda_d$ must be even, say $\lambda_d=2j$ we set the smallest part of $\pi_1$ as 2 and the smallest part of $\pi_2$ as $2j-2=\lambda_d-2$. If $\lambda_{d-1}$ is odd, then the next part of $\pi_1$ is 3 and that of $\pi_2$ is $\lambda_{d-1}-3$.
Observe that $\lambda_{d-1}-3\ge\lambda_d-2$ since $\lambda_{d-1}>\lambda_d$. If $\lambda_{d-1}$ is even, then the next part of $\pi_1$ is 4 and that of $\pi_2$ is $\lambda_{d-1}-4$.
Here $\lambda_{d-1}-4\ge\lambda_d-2$ since $\lambda_{d-1}-\lambda_d\ge 2$. We proceed thus up through the parts of $\pi$. At each stage there is exactly one choice for the part of $\pi_1$ depending on the parity of the original part, and thus
the resulting part of $\pi_2$ is also unique.
\end{proof}

One can notice that $\pi$ and its irreducible partition $\bar \pi$ have the same weight:
$$
\phi (n, d, \pi)=\phi (n, d, \bar \pi),
$$
since they have the same number of odd parts.

Let us denote by $s(d, n)$ the generating function for those partitions in $\bar {\mathcal D}$ that have exactly $d$ parts and largest part equal $n$.

\begin{example} Thus
\begin{align*}
s(5, 8) =& x^2q^{23}+x^2q^{25}+x^2q^{27}\\
=& x^{5-1-2\cdot 1}q^{2+3+4+6+8}+x^{5-1-2\cdot 1}q^{2+4+5+6+8}+ x^{5-1-2\cdot 1}q^{2+4+6+7+8}.
\end{align*}
\end{example}
Let us recall the Gaussian polynomials or $q$-binomial coefficients:
\begin{equation*}
\left[
\begin{array}{l}
A\\
B
\end{array}
\right]_q=
\begin{cases}
 0, & \text{if } B<0 \text{ or } B>A\\
\frac{(q;q)_A}{(q;q)_B(q;q)_{A-B}}, & 0\le B\le A
\end{cases}
\end{equation*}
and $(x;q)_N=(1-x)(1-xq)\cdots(1-xq^{N-1}).$

\begin{lemma}\label{lemma:sdn} The functions $s(d, m)$ can be expressed as follows, depending on parity of $m$:
\begin{itemize}
\item[a)]\begin{equation}\label{eq:d2n}
s(d, 2n)= x^{2n-d-1}q^{2n^2-2dn-n+d^2+2d}\left[
\begin{array}{c}
n-1\\
2n-d-1
\end{array}
\right]_{q^2};
\end{equation}
\item[b)]\begin{equation}\label{eq:d2n1}
s(d, 2n+1)=  x^{2n-d}q^{2n^2-2dn+d^2+3n}\left[
\begin{array}{c}
n-1\\
2n-d
\end{array}
\right]_{q^2}.
\end{equation}
\end{itemize}
\end{lemma}

\begin{proof} We begin by noting initial values for $s(d, n)$. First
$$
s(1,2)=q^2,
$$
otherwise for $n\ne 2$:
$$
s(1, n)=0.
$$
We see that these initial values hold for the right hand sides of \refeq{eq:d2n} and   \refeq{eq:d2n1} as well. Thus, the lemma is proved for $d=1$ and all $n$.

For $d>1$ we have the following  recurrences for $s(d, n)$:
\begin{equation}\label{eq:2d2n}
s(d, 2n)= q^{2n}(s(d-1, 2n-1)+xs(d-1, 2n-2)),
\end{equation}
\begin{equation}\label{eq:2d2n1}
s(d, 2n+1)= q^{2n+1}s(d-1, 2n).
\end{equation}
These recurrences are immediate. The largest part is supplied by $q^{2n}$ in  \refeq{eq:2d2n} and by $q^{2n+1}$ in
\refeq{eq:2d2n1}. If the largest part is $2n$ then the next largest part must be either $2n-1$ or $2n-2$. If the largest part is $2n+1$ then the next largest part must be $2n$.

We must also keep track of how the weights change. Suppose $\bar \pi$ is a partition treated by $s(d-1, 2n-1)$ and its weight is $(d-1)-2s$;
when we attach $2m$ to $\bar \pi$, the weight will become $d-1-2s$, i.e. it does not change. Suppose $\bar \pi$ is a partition treated by $s(d-1, 2n-2)$ and its weight is $(d-1)-1-2s$;
when we attach $2m$ to $\bar \pi$, the weight will become $d-1-2s$, i.e. the weight has increased by $1$, thus accounting for the $x$ in the second term in \eqref{eq:2d2n}. Finally, suppose $\bar \pi$ is a partition treated by $s(d-1, 2n)$ and its weight is $(d-1)-1-2s$.
By attaching $2n+1$ to $\bar \pi$, the weight  becomes $d-2(s+1)$, i.e. the weight has not changed.

To conclude our proof by induction,  we must show that the right hand sides of \refeq{eq:d2n} and   \refeq{eq:d2n1} satisfy these same recurrences.

Namely for \refeq{eq:2d2n1} we see that replacing $d$ by $d-1$ in the right side of \refeq{eq:d2n} and multiplying the result by $q^{2n+1}$ yields
the right side of  \refeq{eq:d2n1}. For \refeq{eq:2d2n1} we must evaluate
\begin{align*}
&q^{2n}\Big(x^{2(n-1)-(d-1)} q^{2(n-1)^2-2(d-1)(n-1)+(d-1)^2+3(n-1)}\left[
\begin{array}{c}
n-2\\
2n-d-1
\end{array}
\right]_{q^2}
\\
&+
x\cdot x^{2(n-1)-(d-1)-1}q^{2(n-1)^2-2(d-1)(n-1)-(n-1)+(d-1)^2+2(d-1)}\left[
\begin{array}{c}
n-2\\
2n-d-2
\end{array}
\right]_{q^2}
\Big)\\
&=x^{2n-d-1}q^{2n^2-2dn-n+d^2+2d}\Big(q^{2(2n-d-1)}\left[
\begin{array}{c}
n-2\\
2n-d-1
\end{array}
\right]_{q^2}+\left[
\begin{array}{c}
n-2\\
2n-d-2
\end{array}
\right]_{q^2}
\Big)\\
&=x^{2n-d-1}q^{2n^2-2dn-n+d^2+2d}\left[
\begin{array}{c}
n-1\\
2n-d-1
\end{array}
\right]_{q^2},
\end{align*}
by \cite{Andrews}, p. 35, eq. (3.34). Thus, both recurrences  \refeq{eq:2d2n} and  \refeq{eq:2d2n1} are satisfied by the right hand sides of \refeq{eq:d2n} and   \refeq{eq:d2n1}. Finally, we see that \refeq{eq:d2n} and   \refeq{eq:d2n1} are valid by mathematical induction on $d$.
\end{proof}

Now we can formulate and prove the main result in this Section.
\begin{theorem}\label{th:generatingE} The generating function for the weighted Euclidean billiard partitions has the following formula:
\begin{equation*}
1+\sum_{n\ge1, m\ge 0}p_{\mathcal D}(m, n)q^n =1+\sum_{d=1}^\infty \sum_{n=0}^\infty \frac{s(d, n)}{(q^2;q^2)_d},
\end{equation*}
where
\begin{equation*}
s(d, 2n)= x^{2n-d-1}q^{2n^2-2dn-n+d^2+2d}\left[
\begin{array}{c}
n-1\\
2n-d-1
\end{array}
\right]_{q^2};
\end{equation*}
\begin{equation*}
s(d, 2n+1)= x^{2n-d}q^{2n^2-2dn-n+d^2+3n}\left[
\begin{array}{c}
n-1\\
2n-d
\end{array}
\right]_{q^2}.
\end{equation*}

\end{theorem}

\begin{proof}
The formulas for $s(d, n)$ have already appeared in Lemma \ref{lemma:sdn}. To complete the proof we note that by the Corollary \ref{corollary:decomp}, the generation function for the partitions in $\mathcal D$ with $d$ parts and larges part $n$ is given by the product
$$
s(d, n)\frac{1}{(1-q^2)(1-q^4)\cdots (1-q^{2d})}=\frac{s(d,n)}{(q^2;q^2)_d},
$$
where the product
$$
\frac{1}{(q^2;q^2)_d}
$$
 generates the general partition into at most $d$ even parts.
\end{proof}

We note that the series given at the beginning of this Section follows directly from Theorem \ref{th:generatingE}, just by setting $x=1$, because
$$
\sum_{m\ge 0|} p_{\mathcal D}(m, n)=p_{\mathcal D}(n).
$$
In the next Section our interest concerns the instance $x=1$ in Theorem \ref{th:generatingE}.

\section{Quadrics and billiards in pseudo-Euclidean spaces and related partitions}\label{section:PE}
	
We consider pseudo-Euclidean space $\mathbb {R}^{k,l}$, with $d=k+l$.
For a given set of positive constants $a_1$, $a_2$, \dots, $a_d$, an
ellipsoid is given by:
\begin{equation}\label{eq:ellipsoid}
\E :\
\frac{x_1^2}{a_1}+\frac{x_2^2}{a_2}+\dots+\frac{x_d^2}{a_d}=1.
\end{equation}
Let us remark that equation of any ellipsoid in the pseudo-Euclidean
space can be brought into the canonical form (\ref{eq:ellipsoid})
using transformations that preserve the pseudo-Euclidean scalar product.

The family of quadrics confocal with $\mathcal{E}$ is:
\begin{equation}\label{eq:confocal}
\Q_{\lambda}\ :\
\frac{x_1^2}{a_1-\lambda} +\dots+ \frac{x_k^2}{a_k-\lambda} +
\frac{x_{k+1}^2}{a_{k+1}+\lambda} +\dots+
\frac{x_d^2}{a_d+\lambda}=1,\qquad\lambda\in\mathbf{R}.
\end{equation}

Unless stated differently, we are going to consider the
non-degenerate case, when set
$$\{a_1,\dots,a_k,-a_{k+1},\dots,-a_{d}\}$$ consists of $d$ different
values:
$$
a_1>a_2>\dots>a_k>0>-a_{k+1}>\dots>-a_d.
$$

For $\lambda\in\{a_1,\dots,a_k,-a_{k+1},\dots,-a_{d}\}$, the quadric
$\Q_{\lambda}$ is degenerate and it coincides with the
corresponding coordinate hyper-plane.

It is natural to join one more degenerate quadric to the family
(\ref{eq:confocal}): the one corresponding to the value
$\lambda=\infty$, that is the hyper-plane at the infinity. The pseudo-Euclidean version of Audin's Lemma \ref{lemma:audin} is the following result
from \cite{DragRadn2012adv}:

\begin{lemma}\label{th:parametri.kaustike}[\cite{DragRadn2012adv}]
	In pseudo-Euclidean space $\mathbf{E}^{k,l}$ consider a line intersecting ellipsoid $\E$ (\ref{eq:ellipsoid}).
	Then this line is touching $d-1$ quadrics from (\ref{eq:confocal}).
	If we denote their parameters by $\alpha_1$, \dots, $\alpha_{d-1}$ and take:
	\begin{gather*}
	\{b_1,\ \dots,\ b_p,\ c_1,\ \dots,\ c_q\}=\{\varepsilon_{1}a_1,\ \dots,\ \varepsilon_{d}a_d,\ \alpha_1,\ \dots,\ \alpha_{d-1}\},\\
	c_q\le\dots\le c_2\le c_1<0<b_1\le b_2\le\dots\le b_p,\quad p+q=2d-1,
	\end{gather*}
	we will additionally have:
	\begin{itemize}
		\item
		if the line is space-like, then $p=2k-1$, $q=2l$, $a_1=b_p$, $\alpha_i\in\{b_{2i-1},b_{2i}\}$ for $1\le i\le k-1$, and $\alpha_{j+k-1}\in\{c_{2j-1},c_{2j}\}$ for $1\le j\le l$;
		\item
		if the line is time-like, then $p=2k$, $q=2l-1$, $c_q=-a_d$, $\alpha_i\in\{b_{2i-1},b_{2i}\}$ for $1\le i\le k$, and $\alpha_{j+k}\in\{c_{2j-1},c_{2j}\}$ for $1\le j\le l-1$;
		\item
		if the line is light-like, then $p=2k$, $q=2l-1$, $b_p=\infty=\alpha_k$, $b_{p-1}=a_1$, $\alpha_i\in\{b_{2i-1},b_{2i}\}$ for $1\le i\le k-1$, and $\alpha_{j+k}\in\{c_{2j-1},c_{2j}\}$ for $1\le j\le l-1$.
	\end{itemize}
	Moreover, for each point on $\ell$ inside $\E$, there is exactly $d$ distinct quadrics from (\ref{eq:confocal}) containing it.
	More precisely, there is exactly one parameter of these quadrics in each of the intervals:
	$$
	(c_{2l-1},c_{2l-2}),\ \dots,\ (c_3,c_2),\ (c_1,0),\ (0,b_1),\ (b_2,b_3),\ \dots,\ (b_{2k-2},b_{2k-1}).
	$$
\end{lemma}

See \cite{DragRadn2012adv}, \cite{DragRadn2019}, \cite{ADR2019rcd} for more details about billiards in pseudo-Euclidean spaces.

We will distinguish tree types of partitions related to billiards in pseudo-Euclidean space $\mathbb R^{k,l}$, the space-type, the time-type,
and the light-type partitions. All of them are of the form  
$$(m_1, m_2,\dots, m_l|n_1, n_2, \dots, n_k),$$ 
such that
$$m_1> m_2>\dots> m_l,$$
and
$$n_1> n_2, \dots> n_k,$$
while there are no assumed relations between $n_i$ and $m_j$.

Let $\mathcal {SD}$  denote the set of all such {\it space-type integer partitions} into  parts
$$(m_1, m_2,\dots, m_l|n_1, n_2, \dots, n_k),$$
where
\begin{itemize}
\item [(S1)] the smallest part $n_k$ is even;
\item [(S2)] adjacent parts $m_i, m_{i+1}$ are never both odd as well as adjacent parts $n_j, n_{j+1}$.
\end{itemize}
Let $p_{\mathcal {SD}}(n)$ denote the number of space-type partitions of $n=n_1+m_1$ that are in $\mathcal {SD}$.

Let $\mathcal {TD}$  denote the set of all such {\it time-type integer partitions} into  parts
$$(m_1, m_2,\dots, m_l|n_1, n_2, \dots, n_k),$$
where
\begin{itemize}
\item [(T1)] the smallest part $m_l$ is even;
\item [(T2)] adjacent parts $m_i, m_{i+1}$ are never both odd as well as adjacent parts $n_j, n_{j+1}$.
\end{itemize}
Let $p_{\mathcal {TD}}(n)$ denote the number of space-type partitions of $n=n_1+m_1$ that are in $\mathcal {TD}$.

Let $\mathcal {LD}$  denote the set of all such {\it light-type integer partitions} into  parts
$$(m_1, m_2,\dots, m_l|n_1, n_2, \dots, n_k),$$
where
\begin{itemize}
\item [(L1)] the smallest parts $m_l$ and $n_k$ are even;
\item [(L2)] adjacent parts $m_i, m_{i+1}$ are never both odd as well as adjacent parts $n_j, n_{j+1}$.
\end{itemize}
Let $p_{\mathcal {LD}}(n)$ denote the number of space type partitions of $n=n_1+m_1$ that are in $\mathcal {LD}$.

Let us denote by $\tilde s(n_1, k)$ the generation function  of reduced partitions
$$n_1> n_2>\dots> n_k,$$
such that any two adjacent parts  $n_i, n_{i+1}$ are never both odd, but without assuming anything about the parity of $n_k$.
In the sequel, we use the instance $x=1$ of Theorem \ref{th:generatingE} and of the notions from that statement.
\begin{lemma} The generating function $\tilde s $ satisfies the following
$$
\tilde s(n_1, k)= s(n_1, k) + qs(n_1, k-1).
$$
\end{lemma}
\begin{proof}
We consider two cases depending on parity of $n_k$. If $n_k$ is even, than the partition belongs to $s(n_1, k)$. If $n_k$ is odd, it has to be $n_k=1$ and $n_{k-1}=2$. Thus the partition $n_1> n_2>\dots> n_{k-1}$ belongs to $s(n_1, k-1)$.
\end{proof}

\begin{theorem}\label{th:generatingPE} The generating functions for the pseudo-Euclidean billiard partitions have the following formulas:
\begin{itemize}
\item [(S)] for the space-type:
\begin{equation*}
1+\sum_{n\ge1}p_{\mathcal SD}(n)q^n =1+\sum_{k=1}^\infty \sum_{l=1}^\infty \sum_{n_1=0}^\infty \sum_{m_1=0}^\infty\frac{s(n_1, k)}{(q^2;q^2)_k}\frac{\tilde s(m_1, l)}{(q^2;q^2)_l};
\end{equation*}
\item [(T)] for the space-type:
\begin{equation*}
1+\sum_{n\ge1}p_{\mathcal TD}(n)q^n =1+\sum_{k=1}^\infty \sum_{l=1}^\infty \sum_{n_1=0}^\infty \sum_{m_1=0}^\infty\frac{\tilde s(n_1, k)}{(q^2;q^2)_k}\frac{s(m_1, l)}{(q^2;q^2)_l};
\end{equation*}
\item [(L)] for the space-type:
\begin{equation*}
1+\sum_{n\ge1}p_{\mathcal SD}(n)q^n =1+\sum_{k=1}^\infty \sum_{l=1}^\infty \sum_{n_1=0}^\infty \sum_{m_1=0}^\infty\frac{s(n_1, k)}{(q^2;q^2)_k}\frac{s(m_1, l)}{(q^2;q^2)_l}.
\end{equation*}
\end{itemize}
\end{theorem}

The first author is preparing a paper that extensively develops the partition methods used in section 2.

\subsection*{Acknowledgment}
The research of V.~D.~and M.~R.~was supported by the Discovery Project \#DP190101838 \emph{Billiards within confocal quadrics and beyond} from the Australian Research Council and Project \#174020 \emph{Geometry and Topology of Manifolds, Classical Mechanics and Integrable Systems} of the Serbian Ministry of Education, Technological Development and Science. V.~D.~would like to thank Sydney Mathematics Research Institute and their International Visitor Program for kind hospitality.

\end{document}